\documentclass{amsart}
\usepackage{amsfonts}

\newtheorem{theorem}{Theorem}
\theoremstyle{plain}

\newtheorem{remark}{Remark}

\numberwithin{equation}{section}
\input{tcilatex}

\begin{document}
\title[The Helix Relation Between Two Curves ]{The Helix Relation Between
Two Curves }
\author{\.{I}lkay ARSLAN G\"{U}VEN}
\address{ }
\email{}
\urladdr{}
\author{Yusuf YAYLI}
\curraddr{}
\email{ }
\urladdr{}
\date{}

\begin{abstract}
In this study, we give the relation of being general helix and slant helix
of two curves by using the equation between them. Also we find some results
and express the characterizations of these curves.
\end{abstract}

\maketitle

\section{Introduction}

The characterization of general helix is given in Lancret's Theorem in 1802.
The general helix and the associated plane curve in the Euclidean 3-space $%
E^{3}$ were studied in \cite{Sy}. It was given clearly that how to convert
the associated plane curve to the general helix and vice versa. The
equations of the general helix and associated plane curve were given to make
this convertion. Also in \cite{Izu}, a similiar equation was given to show
that cylindrical helices can be constructed from plane curves.

In \cite{Guv}, the general helix and associated plane curve were studied in
Minkowski 3-space $E_{1}^{3}$. The new equations which say how to obtain the
general helix from the plane curve were given in $E_{1}^{3}$.

In this paper, we take the equation of two curves which is similiar to the
equation in \cite{Sy}. Then we calculate the Frenet vectors and axis of
symmetry of each curve and obtain the relation between them of how to being
general helix, slant helix and Bertrand mates.

\section{Preliminaries}

We now recall some basic notions about classical differential geometry of
space curves in Euclidean space $E^{3}$.

Let $\beta :I\longrightarrow \mathbb{R}^{3}$ be a curve with arc-length
parameter $s$ and let $\left\{ T,N,B\right\} $ denote the Frenet frame of $%
\beta $. $T(s)=\beta ^{\prime }(s)$ is called the \textit{unit tangent vector%
} of $\beta $ at $s$. The \textit{curvature} of $\beta $ is given by $\kappa
(s)=\left\Vert \beta ^{\prime \prime }(s)\right\Vert $. The \textit{unit
principal normal vector }$N(s)$ of $\beta $ at $s$ is given by $\beta
^{\prime \prime }(s)=\kappa (s).N(s)$. Also the unit vector $B(s)=T(s)\times
N(s)$ is called the \textit{unit binormal vector }of $\beta $ at $s$. Then
the famous Frenet formula holds as;%
\begin{eqnarray*}
T^{\prime }(s) &=&\kappa (s).N(s) \\
N^{\prime }(s) &=&-\kappa (s).T(s)+\tau (s)B(s) \\
B^{\prime }(s) &=&-\tau (s)N(s)
\end{eqnarray*}%
where $\tau (s)$ is the \textit{torsion} of $\beta $ at $s$.

Also the Frenet vectors of a curve $\alpha $, which is not given by
arc-length parameter can be calculated as;%
\begin{equation}
T(t)=\frac{\alpha ^{\prime }(t)}{\left\Vert \alpha ^{\prime }(t)\right\Vert }%
\text{ \ \ \ , \ \ \ \ }B(t)=\frac{\alpha ^{\prime }(t)\times \alpha
^{\prime \prime }(t)}{\left\Vert \alpha ^{\prime }(t)\times \alpha ^{\prime
\prime }(t)\right\Vert }\text{ \ \ \ , \ \ \ }N(t)=B(t)\times T(t).  \tag{1}
\end{equation}

A curve $\beta :I\longrightarrow \mathbb{R}^{3}$ is called a general helix
if its tangent line forms a constant angle with a fixed straight line. This
straight line is the axis of general helix. A classical result \ stated by
Lancret says that "a curve is a general helix if and only if \ the ratio of
the curvature to torsion is constant". If both curvature and torsion are
non-zero constant, it is of course a general helix which is called circular
helix.

A slant helix in $E^{3}$ is defined by the property that the principal
normal line makes a constant angle with a fixed direction. In \cite{Kula},
it is shown that $\alpha $ is a slant helix in $E^{3}$ if and only if the
geodesic curvature of the principal normal of the space curve $\alpha $ is a
constant function.

Let two curves be $\alpha $ and $\beta $ in $E^{3}$. They are called
Bertrand curves if their principal normal vectors are linearly dependent. We
say that $\alpha $ and $\beta $ are Bertrand mates.

\section{The Equation Of Two Curves}

Let $\overline{\alpha }$ be a curve and $\alpha $ be a unit speed general
helix in $E^{3}$. $s$, will denote arc-lengthparameter of $\alpha $. The
Frenet frame of $\overline{\alpha }$ and $\alpha $ are indicated by $\left\{ 
\overline{T},\overline{N},\overline{B}\right\} $ and $\left\{ T,N,B\right\} $%
, respectively. The curvatures of $\overline{\alpha }$ and $\alpha $ are $%
\overline{\kappa }$ and $\kappa $; the torsions are $\overline{\tau }$ and $%
\tau $. $a$ is the constant axis of general helix, $\theta $ is the angle
between $a$ and $T$. The axis is given by%
\begin{equation*}
a=\cos \theta T+\sin \theta B.
\end{equation*}

The equation between $\overline{\alpha }$ and $\alpha $ is denoted in \cite%
{Sy} as%
\begin{equation*}
\overline{\alpha }(s)=\alpha (s_{0})+\sin \theta .\alpha (s)+a.(s-s_{0})\cos
\theta
\end{equation*}%
where $\alpha (s_{0})$ and $s_{0}$ are arbitrary constant vector and point.

Now, let find the Frenet vectors $\left\{ \overline{T},\overline{N},%
\overline{B}\right\} $ of $\overline{\alpha }$. Since $s$ is not arc-length
parameter of $\overline{\alpha }$, we use the equations in (1).

The tangent vector $\overline{T}$ of $\overline{\alpha }$ is

\begin{equation*}
\overline{T}(s)=\frac{\sin \theta +\cos ^{2}\theta }{\sqrt{1+\cos \theta
.\sin 2\theta }}T(s)+\frac{\cos \theta .\sin \theta }{\sqrt{1+\cos \theta
.\sin 2\theta }}B(s),
\end{equation*}

\bigskip

the principal normal vector $\overline{N}$ of $\overline{\alpha }$ is

\begin{equation*}
\overline{N}(s)=\left[ \frac{\mu }{\sqrt{\lambda ^{2}+\mu ^{2}}}.\frac{\sin
\theta +\cos ^{2}\theta }{\sqrt{1+\cos \theta .\sin 2\theta }}-\frac{\lambda 
}{\sqrt{\lambda ^{2}+\mu ^{2}}}.\frac{\cos \theta .\sin \theta }{\sqrt{%
1+\cos \theta .\sin 2\theta }}\right] N(s)
\end{equation*}

\bigskip

where 
\begin{eqnarray*}
\lambda &=&\cos \theta .\sin ^{2}\theta +\cos ^{3}\theta .\sin \theta \\
\mu &=&(\sin \theta +\cos ^{2}\theta )\kappa -(\cos \theta .\sin ^{2}\theta
+\cos ^{3}\theta .\sin \theta )\tau .
\end{eqnarray*}%
Here if we say 
\begin{equation*}
\frac{\mu }{\sqrt{\lambda ^{2}+\mu ^{2}}}.\frac{\sin \theta +\cos ^{2}\theta 
}{\sqrt{1+\cos \theta .\sin 2\theta }}-\frac{\lambda }{\sqrt{\lambda
^{2}+\mu ^{2}}}.\frac{\cos \theta .\sin \theta }{\sqrt{1+\cos \theta .\sin
2\theta }}=c
\end{equation*}%
then we can denote the principal normal vectors as $\overline{N}(s)=c.N(s)$.

Also the binormal vector $\overline{B}$ of $\overline{\alpha }$ is found as

\begin{equation*}
\overline{B}(s)=\frac{\lambda }{\sqrt{\lambda ^{2}+\mu ^{2}}}T(s)+\frac{\mu 
}{\sqrt{\lambda ^{2}+\mu ^{2}}}B(s).
\end{equation*}

\bigskip

We will state the following theorems whose proofs will be done by these
calculations.

\begin{theorem}
Let a curve and a general helix be $\overline{\alpha }$ and $\alpha $,
respectively. The equation between them is given by $\overline{\alpha }%
(s)=\alpha (s_{0})+\sin \theta .\alpha (s)+a.(s-s_{0})\cos \theta $ \ where $%
a$ is the axis of general helix. If $\alpha $ is general helix then $%
\overline{\alpha }$ is a general helix.
\end{theorem}

\begin{proof}
Let the curve $\alpha $ be a general helix. The tangent and binormal vectors 
$T$ and $B$ of $\alpha $ make constant angle with a constant vector which
can be $a$, the axis of $\alpha $.

\begin{eqnarray*}
\langle a,T\rangle &=&\cos \theta \\
\langle a,B\rangle &=&\sin \theta
\end{eqnarray*}
Since the tangent vector $\overline{T}$ of $\overline{\alpha }$ depends on $%
T $ and $B$, it also make constant angle with that constant vector.

\begin{equation*}
\langle a,\overline{T}\rangle =\frac{\cos \theta .\sin \theta +\cos
^{3}\theta }{\sqrt{1+\cos \theta .\sin 2\theta }}+\frac{\cos \theta .\sin
^{2}\theta }{\sqrt{1+\cos \theta .\sin 2\theta }}
\end{equation*}%
So $\overline{\alpha }$ is a general helix.
\end{proof}

\begin{theorem}
Let a curve and a general helix be $\overline{\alpha }$ and $\alpha $,
respectively. The equation between them is given by $\overline{\alpha }%
(s)=\alpha (s_{0})+\sin \theta .\alpha (s)+a(s-s_{0})\cos \theta $ \ where $%
a $ is the axis of general helix. If the number $c$ between the principal
normal vectors of $\overline{\alpha }$ and $\alpha $ is constant, then $%
\alpha $ is slant helix if and only if $\overline{\alpha }$ is a slant helix.
\end{theorem}

\begin{proof}
The relation of principal normal vectors $\overline{N}$ and $N$ of $%
\overline{\alpha }$ and $\alpha $ is calculated as;%
\begin{equation*}
\overline{N}(s)=\left[ \frac{\mu }{\sqrt{\lambda ^{2}+\mu ^{2}}}.\frac{\sin
\theta +\cos ^{2}\theta }{\sqrt{1+\cos \theta .\sin 2\theta }}-\frac{\lambda 
}{\sqrt{\lambda ^{2}+\mu ^{2}}}.\frac{\cos \theta .\sin \theta }{\sqrt{%
1+\cos \theta .\sin 2\theta }}\right] N(s)
\end{equation*}%
which we denoted by $\overline{N}(s)=c.N(s)$.

Here if $c$ is constant then $\overline{N}$ and $N$ are linearly dependent.

Firstly let $\alpha $ be a slant helix.Then the principal normal vector of $%
\alpha $ makes a constant angle with a fixed direction. Since the principal
normal vector $\overline{N}$ is linearly dependent with $N$, $\overline{N}$
also makes a constant angle with that fixed direction. So $\overline{\alpha }
$ is a slant helix.

The opposite acceptance can be done in the same way.
\end{proof}

\begin{remark}
The number $c$ between the principal normal vectors was taken constant in
the theorem. This number $c$ is constant under the condition of $\kappa $
and $\tau $ are constant. Thus $\alpha $ is a circular helix. If $\alpha $
is a circular helix, then $\alpha $ is slant helix if and only if $\overline{%
\alpha }$ is a slant helix.
\end{remark}

\begin{theorem}
Let a curve and a general helix be $\overline{\alpha }$ and $\alpha $,
respectively. The equation between them is given by $\overline{\alpha }%
(s)=\alpha (s_{0})+\sin \theta .\alpha (s)+a(s-s_{0})\cos \theta $ \ where $%
a $ is the axis of general helix. If the number $c$ between the principal
normal vectors of $\overline{\alpha }$ and $\alpha $ is constant, then $%
\overline{\alpha }$ and $\alpha $ are Bertrand mates.
\end{theorem}

\begin{proof}
In the equation $\overline{N}(s)=c.N(s)$ , let $c$ be a constant number,
then the principal normal vectors of $\overline{\alpha }$ and $\alpha $ are
linearly dependent. So $\overline{\alpha }$ and $\alpha $ are Bertrand
curves.
\end{proof}

\begin{remark}
When $c$ is taken as a constant, then $\alpha $ is a circular helix. If $%
\alpha $ is a circular helix, then $\ \overline{\alpha }$ and $\alpha $ are
Bertrand mates.
\end{remark}

Now we will give an example .

\textbf{Example: \ }Let $\alpha (s)=(6s,3s^{2},s^{3})$ be a general helix
with the curvature and torsion;%
\begin{equation*}
\kappa (s)=\frac{2}{3(s^{2}+2)}\text{ \ \ \ \ \ , \ \ \ \ \ \ }\tau (s)=%
\frac{2}{3(s^{2}+2)}.
\end{equation*}%
The axis of $\alpha $ is calculated by $a(s)=\cos \theta T(s)+\sin \theta
B(s)$. Here the angle $\theta ,$ between $a$ and $T$ is 
\begin{equation*}
\frac{\kappa (s)}{\tau (s)}=\tan \theta =1\text{ \ \ }\Longrightarrow \text{
\ \ }\theta =\frac{\pi }{4}.
\end{equation*}%
If the vectors $\left\{ T,N,B\right\} $ are Frenet vectors of $\alpha $,
then 
\begin{eqnarray*}
T(s) &=&\left( \frac{2}{s^{2}+2},\frac{2s}{s^{2}+2},\frac{s^{2}}{s^{2}+2}%
\right) \\
N(s) &=&\left( \frac{-2s^{3}-4s}{(s^{2}+2)^{2}},\frac{s^{4}-4s^{2}-8}{%
(s^{2}+2)^{2}},\frac{2s^{3}+4s}{(s^{2}+2)^{2}}\right) \\
B(s) &=&\left( \frac{s^{2}}{s^{2}+2},\frac{-2s}{s^{2}+2},\frac{2}{s^{2}+2}%
\right) .
\end{eqnarray*}%
So the axis $a(s)$ is found as%
\begin{equation*}
a(s)=\left( \frac{2\sqrt{2}+\sqrt{2}s^{2}}{s^{2}+2},0,\frac{2\sqrt{2}+\sqrt{2%
}s^{2}}{s^{2}+2}\right) .
\end{equation*}%
The curve $\overline{\alpha }(s)$ is 
\begin{equation*}
\overline{\alpha }(s)=\left( \frac{(3\sqrt{2}+1)s^{3}+(6\sqrt{2}+2)s}{s^{2}+2%
},\frac{3\sqrt{2}}{2}s^{2},\frac{\frac{\sqrt{2}}{2}s^{5}+(\sqrt{2}+1)s^{3}+2s%
}{s^{2}+2}\right)
\end{equation*}%
by taking the arbitrary constant vector $\alpha (s_{0})=\overrightarrow{0}$
and arbitrary point $s_{0}=0.$

Then the Frenet vectors of $\overline{\alpha }$ are calculated by using $%
T,N,B;$%
\begin{eqnarray*}
\overline{T}(s) &=&\frac{1}{\sqrt{4+2\sqrt{2}}}\left( 1+\frac{2\sqrt{2}}{%
s^{2}+2},\frac{2\sqrt{2}s}{s^{2}+2},1+\frac{\sqrt{2}s^{2}}{s^{2}+2}\right) \\
&& \\
\overline{N}(s) &=&\frac{4+2\sqrt{2}-3(s^{2}+2)^{2}}{\sqrt{4+2\sqrt{2}}.%
\sqrt{9(s^{2}+2)^{4}+8}}.\left( \frac{-2s^{3}-4s}{(s^{2}+2)^{2}},\frac{%
s^{4}-4s^{2}-8}{(s^{2}+2)^{2}},\frac{2s^{3}+4s}{(s^{2}+2)^{2}}\right) \\
&& \\
\overline{B}(s) &=&\frac{1}{\sqrt{9(s^{2}+2)^{4}+8}}\left( 6(s^{2}+2)+\frac{2%
\sqrt{2}s^{2}}{s^{2}+2},6s(s^{2}+2)-\frac{4\sqrt{2}s}{s^{2}+2}%
,3s^{2}(s^{2}+2)-\frac{4\sqrt{2}}{s^{2}+2}\right) .
\end{eqnarray*}

\bigskip

\.{I}lkay Arslan G\"{u}ven

University of Gaziantep

Department of Mathematics

\c{S}ehitkamil, 27310, Gaziantep, Turkey

E-mail: iarslan@gantep.edu.tr, ilkayarslan81@hotmail.com

\bigskip

Yusuf Yayl\i

Ankara University

Department of Mathematics

Ankara, Turkey

E-mails: Yusuf.Yayli@science.ankara.edu.tr

\end{document}